\documentclass{article}
\usepackage[utf8]{inputenc}
\pdfoutput=1

\usepackage[margin=1.25in]{geometry}

\usepackage{amsmath,color,graphicx,amssymb}
\usepackage{amsthm}
\usepackage{amsmath}
\usepackage{tikz}
\usetikzlibrary{calc,matrix,arrows,shapes.geometric,positioning,decorations.pathmorphing}
\usepackage{tikz-cd}

\usepackage{float}
\usepackage{slashbox}
\usepackage{indentfirst}
\usepackage{fancyhdr}
\usepackage{enumitem}

\DeclareMathOperator{\Tor}{Tor}
\DeclareMathOperator{\Ext}{Ext}

\newtheorem{thm}{Theorem}[section] 

\newtheorem{lemma}[thm]{Lemma}     

\newtheorem{cor}[thm]{Corollary}

\newtheorem{prop}[thm]{Proposition}

\theoremstyle{definition}
\newtheorem{defn}[thm]{Definition}

\theoremstyle{definition}

\theoremstyle{definition}
\newtheorem{remk}[thm]{Remark}

\theoremstyle{definition}

\theoremstyle{definition}

\theoremstyle{definition}

\theoremstyle{definition}

\allowdisplaybreaks

\title{\large{\uppercase{A self-dual complete resolution}}}
\author{\normalsize{\uppercase{Rachel N. Diethorn}}}
\date{}

\begin{document}

\maketitle

\begin{abstract}
We construct a self-dual complete resolution of a module defined by a pair of embedded complete intersection ideals in a local ring.  Our construction is based on a gluing construction of Herzog and Martsinkovsky and exploits the structure of Koszul homology in the embedded complete intersection case.  As a consequence of our construction, we produce an isomorphism between certain stable homology and cohomology modules.
\end{abstract}

\vspace{0.3cm}


\section{Introduction}

The main motivation for studying self-duality of complete resolutions is the following classical result of Buchsbaum and Eisenbud on self-duality of finite free resolutions.

\begin{thm}(\cite[Theorem 1.5]{BE77})
If $R$ is a local ring and $I\subseteq R$ is a Gorenstein ideal, then the minimal free resolution of $R/I$ over $R$ is self-dual. 
\end{thm}

For complete resolutions, the idea of self-duality is not so well understood.  However, a related, but weaker notion of symmetric growth in complete resolutions has been a topic of interest in recent years.  In \cite[Theorem 5.6]{AB2000}, Avramov and Buchweitz prove that over a complete intersection ring, the growth of the Betti sequence of a finitely generated module is the same as the growth of its Bass sequence.  In particular, every minimal complete resolution over a complete intersection ring exhibits symmetric growth; however, this is not the case over an arbitrary Gorenstein ring.  In \cite[Theorem 3.2]{JS2005}, Jorgensen and  {\c{S}}ega construct a complete resolution over a Gorenstein ring that does not grow symmetrically.  On the other hand, they show that minimal complete resolutions over a Gorenstein ring of minimal multiplicity or with low codimension exhibit symmetric growth. In \cite[Corollary 4.3]{BJ2014}, Bergh and Jorgensen examine minimal totally acyclic complexes  and give a criterion under which these complexes have symmetric growth.  In particular, they prove that in a local ring, the minimal complete resolution of a module of finite complete intersection dimension exhibits symmetric growth.  

Given these results, it seems natural to ask which modules have complete resolutions that not only exhibit symmetric growth, but are actually self-dual.  In this paper, we give a class of modules over a local ring which have self-dual complete resolutions.  In particular, we construct the minimal complete resolution of a module defined by a pair of embedded complete intersection ideals over a local ring, and show that it is self-dual.  As a consequence, we establish an isomorphism between certain stable homology and cohomology modules. 

Complete resolutions of such modules have appeared previously in the literature; albeit, with different constructions, and from different perspectives. Indeed, the resolution we construct in this paper turns out to be isomorphic to one given by Kustin and {\c{S}}ega in \cite{KS18} and another given by Eisenbud and Schreyer in \cite{ES21}; although, the latter requires additional hypotheses on the ambient ring.  Another construction given by Buchweitz, Pham, and Roberts appears in the Oberwolfach report \cite{ober}, but without proof.  

We give a different construction that becomes useful in exploring self-duality.  Our construction, which is built upon a gluing construction of Herzog and Martsinkovsky in \cite{HM93} and exploits the structure of Koszul homology in the embedded complete intersection case, allows us to describe, explicitly, the maps in the resolution and an isomorphism between the resolution and its dual.        

We now outline the contents of this paper.  In Section 2 we recall the notions of complete resolutions and stable (co)homology, and we introduce a precise notion of self-duality of complete resolutions.  We also produce an isomorphism between stable homology and cohomology modules in the case that such a resolution exists.  In Section 3, we develop a key ingredient we will need in our main construction of the resolution; namely, we provide an explicit description of the generators of Koszul homology modules in the embedded complete intersection case.  In Section 4, we construct the self-dual complete resolution.

\section{Preliminaries} 
Throughout this paper, we let $Q$ be a commutative, Noetherian, local ring, and we use the notation $(\,\underline{\hspace{11pt}}\,)^*:=\mathrm{Hom}_Q(\,\underline{\hspace{11pt}}\,,Q)$.  We begin this section by recalling some basic definitions and facts that we use throughout the paper.  They can also be found in \cite{AB2000}.

\begin{defn}
A \textit{complete resolution} of a finitely generated $Q$-module $M$ is a complex 
\begin{displaymath}
T\colon\thinspace\thinspace \cdots\rightarrow T_2\rightarrow T_1\rightarrow T_0\rightarrow T_{-1}\rightarrow T_{-2}\rightarrow\cdots
\end{displaymath}
of finitely generated free $Q$-modules which satisfies the following conditions:
\begin{enumerate}
\item[(1)] both $T$ and $T^*$ are exact
\item[(2)] for some $r\geq 0$, the truncation $T_{\geq r}$ is isomorphic to $F_{\geq r}$, where $F$ is a free resolution of $M$ over $Q$.
\end{enumerate}
\end{defn}

Next we discuss stable (co)homology.

\begin{defn}  The \textit{stable (co)homology modules} of $M$ and $N$ are defined as
\begin{align*}
\widehat{\mathrm{Tor}}_n^Q(M,N)&:=H_n(T\underset{Q}{\otimes}N) \\
\widehat{\mathrm{Ext}}^n_Q(M,N)&:=H^n(\text{Hom}_Q(T,N))
\end{align*}
where $T$ is a complete resolution of $M$.
\end{defn}

It is well-known that if a module $M$ has complete resolutions, then any two complete resolutions of $M$ are homotopy equivalent.  Thus, for all integers $n$, the stable (co)homology modules are well-defined; see for example \cite[Lemma 2.4]{ck}.

On the other hand, modules do not always have complete resolutions.  In fact, a module has a complete resolution if and only it has finite Gorenstein dimension.   Thus in a Gorenstein ring, every module has a complete resolution; see for example \cite[4.4.4]{AB2000}.

Now we introduce a precise definition for self-duality of complete resolutions, which we use throughout this paper.  

\begin{defn}
 A complete resolution $T$ is \textit{self-dual} if $T^*$ is isomorphic to $T[\,i\,]$ for some integer $i$.
\end{defn}

In other words, we consider a complete resolution that is isomorphic to its dual, up to some shift, to be self-dual.  Modules that have self-dual complete resolutions satisfy nice (co)homological properties, as shown in the next proposition.

\begin{prop}\label{torext}
If a $Q$-module $M$ has a self-dual complete resolution then for any $Q$-module $N$ there is an isomorphism 
\begin{align*}
\widehat{\Tor}_n^Q(M,N)\cong\widehat{\Ext}_Q^{n+i}(M,N)
\end{align*}
for some integer $i$ and for all integers $n$.
\end{prop}

\begin{proof} 
If $T$ is a self-dual resolution of $M$ and $T^*\cong T[\,i\,]$, then we have the following isomorphisms
\begin{align*}
\widehat{\Tor}_n^Q(M,N)&=H_n(T\underset{Q}{\otimes}N)\cong H_n(\text{Hom}_Q(T^*,N))\cong H_n(\text{Hom}_Q(T[\,i\,],N)) \\ 
&\cong H_{n+i}(\text{Hom}_Q(T,N))=\widehat{\Ext}_Q^{n+i}(M,N)
\end{align*} 
where the first isomorphism follows from the standard isomorphism $T\otimes_Q N\cong\text{Hom}_Q(T^*,N)$.
\end{proof}

In Section 4, we establish a family of modules which have self-dual complete resolutions and thus, by Proposition \ref{torext}, they have isomorphic homology and cohomology modules.

\section{Generators of Koszul Homology}

We begin this section by establishing some notation that we use throughout the remainder of this paper.  Let $Q$ be a commutative Noetherian Cohen-Macaulay local ring and fix regular sequences $\underline{f}=f_1,\dots,f_r$  and $\underline{g}=g_1,\dots,g_s$ in $Q$, such that $I:=(\underline{f})\subseteq(\underline{g})$.  In this case, we call $I$ an \textit{embedded complete intersection ideal} in $(\underline{g})$.  Let $R=Q/I$ and $J$ be the ideal $(\underline{g})/I$ in $R$, with $g$ the \textit{grade} of $J$; that is, the length of the longest $R$-regular sequence contained in $J$.  Denote by 
\begin{align*}
K(\underline{g};Q)=Q\langle dg_1,\dots,dg_s |\partial_K(dg_i)=g_i\rangle
\end{align*}
and 
\begin{align*}
K(\underline{f};Q)=Q\langle df_1,\dots,df_r |\partial_K(df_i)=f_i\rangle
\end{align*}
the Koszul complexes on the sequences $\underline{g}$ and $\underline{f}$, respectively, and let $H_i(\underline{g};R)$ be the Koszul homology module $H_i(R\otimes_Q K(\underline{g};Q))$. 
\newline
\indent Since $(\underline{f})\subseteq (\underline{g})$, for each $j=1,\dots,r$, we can write 
\begin{align}\label{aji}
f_j=\sum_{i=1}^s a_{ji}g_i 
\end{align}
where $a_{ji}\in Q$.  Denote by $A$ the $s\times r$ matrix of coefficients $(a_{ji})$ and by $A_{j_1,\dots,j_k}^{i_1,\dots,i_k}$ the $k\times k$ minor of $A$ corresponding to the rows $j_1,\dots,j_k$ and the columns $i_1,\dots,i_k$.  
\newline
\indent In this section we provide explicit descriptions of the generators of the Koszul homology modules $H_i(\underline{g};R)$, which we use in the next section to construct a self-dual complete resolution of $R/J$ over $R$.  We first note that in the case of embedded complete intersection ideals, the structure of $H(\underline{g};R)$ is well-understood in the following sense.  We have that $J\subseteq R$ is a quasi-complete intersection ideal as defined in \cite[1.1]{quasici}.  Indeed, since $(\underline{g})$ is a complete intersection, it is a quasi-complete intersection ideal of $Q$, and thus $J$ is a quasi-complete intersection ideal of $R$ by \cite[Lemma 1.4]{quasici}.  Therefore, $H_\ell(\underline{g};R)=\bigwedge^{\ell}H_1(\underline{g};R)$ is an exterior power of the first Koszul homology.  We use this structure to describe the generators of Koszul homology.  The generators we provide in Lemma \ref{embeddedcigens} are given in terms of the minors of the matrix $A$ defined above, and this plays a crucial role in our constructions in the next section.    
\newline
\indent  We also note that one could instead describe the generators of Koszul homology using \cite[Corollary 3.8]{DIETHORN2020106387};  however, due to the extra structure on Koszul homology in the embedded complete intersection case, the descriptions of the generators we obtain below are simpler.
\newline
\indent
Before stating the lemma, we establish some notation we will use in the proof.  For an element $z$ in $(K(\underline{f};Q)\otimes_Q K(\underline{g};Q))_{\ell}$ we write $z=(z_0,...,z_{\ell})$ where each $z_i$ is an element of $K(\underline{f};Q)_{\ell-i}\otimes K(\underline{g};Q)_i$.  For an element $q\in Q$, we denote its image under the surjection $Q\twoheadrightarrow R$ as $\bar{q}$ and for an element $z$ in $ K_{\ell}(\underline{g};R)$, we denote its homology class in $H_{\ell}(\underline{g};R)$ as $[z]$.

\begin{lemma}\label{embeddedcigens} An $R/J$-basis for $H_\ell(\underline{g};R)$ is given by homology classes of the elements
\begin{displaymath}
\sum_{1\leq i_1<\dots<i_\ell\leq s}\overline{A_{j_1,..,j_\ell}^{i_1,\dots,i_\ell}}dg_{i_1}\dots dg_{i_\ell}
\end{displaymath}
for $1\leq j_1<\dots<j_\ell\leq r$.
\end{lemma}

\begin{proof}
By the discussion above, it suffices to find a basis for $H_1(\underline{g};R)$ and show that the product of each set of $\ell$ such basis elements yields the desired formulas.
\newline
\indent  Consider the following isomorphisms
\begin{align*}
H_{\ell}(\underline{g};R)&\cong H_{\ell}(R\underset{Q}{\otimes}K(\underline{g};Q))\cong\Tor_{\ell}^Q(R,R/J)\cong H_{\ell}(K(\underline{f};Q)\underset{Q}{\otimes}R/J)\cong K_{\ell}(\underline{f};Q)\underset{Q}{\otimes} R/J,
\end{align*}
where the second and third isomorphisms follow from the fact that $\underline{f}$ and $\underline{g}$ are regular sequences and the last follows from the fact that $I\subseteq (\underline{g})$.  The isomorphism above is defined by sending an element $b\otimes 1$ in $K_\ell(\underline{f};Q)\otimes_Q R/J$ to $[\bar{z_\ell}]$ where $z=(z_0,\dots,z_\ell)$ is a cycle in $K(\underline{f};Q)\otimes_Q K(\underline{g};Q)$ such that $z_0$ is sent to $b\otimes 1$ in the surjection 
\begin{align*}
K_\ell(\underline{f};Q)\otimes_Q K_0(\underline{g};Q)\twoheadrightarrow K_\ell(\underline{f};Q)\otimes_Q R/J 
\end{align*} 
and $z_\ell$ is sent to $\bar{z_\ell}$ in the surjection
\begin{align*}
K_0(\underline{f};Q)\otimes_Q K_\ell(\underline{g};Q)\twoheadrightarrow R\otimes_Q K_\ell(\underline{g};Q);
\end{align*}
see for example \cite{herzog}.

Taking $z_0=df_j\otimes 1$ and $z_1=\sum_{i=1}^s a_{ji}\otimes dg_i$, we see that 
\begin{align*}
(1\otimes \partial_K)(z_1)=f_j\otimes 1=(\partial_k\otimes 1)(z_{0})
\end{align*}
for each $i=1,\dots,r$.  Thus, $(z_0,z_1)$ is a cycle, and since the elements $df_j\otimes \bar{1}$ for $j=1,\dots,r$ form a basis of $K_{1}(\underline{f};Q)\otimes_Q R/J$, the homology classes of the elements
\begin{align*}
\sum_{i=1}^s \overline{a_{ji}}dg_i
\end{align*}
for $j=1,\dots,r$ form a basis of $H_1(\underline{g};R)$.  
\newline 
\indent  Multiplying any distinct two of these generators together yields
\begin{align*}
\Big(\sum_{i_1=1}^s \overline{a_{j_1,i_1}}dg_{i_1}\Big)\cdot\Big(\sum_{i_2=1}^s \overline{a_{j_2,i_2}}dg_{i_2}\Big)=\sum_{i_1=1}^s\sum_{i_2=1}^s\overline{a_{j_1,i_1}a_{j_2,i_2}}dg_{i_1}dg_{i_2}=\sum_{1\leq i_1<i_2\leq s}\overline{A_{j_1,j_2}^{i_1,i_2}}dg_{i_1}dg_{i_2}.
\end{align*}
Inductively, we obtain the desired formulas.      
\end{proof}

We end this section by giving a formula for the generator of the first nonvanishing Koszul cohomology module, which we use in our construction in the next section.
  
\begin{cor}\label{cohomgens} The homology class of the element 
\begin{displaymath}
\sum_{1\leq i_1<\dots<i_r\leq s}(-1)^{rg+k_1+\dots+k_g}A_{1,..,r}^{i_1,\dots,i_r}dg^*_{k_1}\dots dg^*_{k_g}
\end{displaymath}
where $\{k_1,\dots,k_g\}$ is the complement of the set $\{i_1,\dots,i_r\}$ in $\{1,\dots,s\}$, generates the Koszul cohomology module $H^g(\underline{g};R)$.
\end{cor}

\begin{proof}
We first note that since $Q$ is Cohen-Macaulay, we have equalities
\begin{align*}
g=\text{ht}J=\text{dim}R-\text{dim}R/J=\text{dim}Q-r-(\text{dim}Q-s)=s-r.
\end{align*}
Now the result follows directly from Lemma \ref{embeddedcigens} and the self-duality isomorphism of the Koszul complex, 
\begin{align*}
K_j(\underline{g};R)&\overset{\cong}\longrightarrow K^{s-j}(\underline{g};R) \\
dg_{i_1}\dots dg_{i_j}&\longmapsto (-1)^p dg^*_{k_1}\dots dg^*_{k_{s-j}}
\end{align*}
where $\{k_1,\dots,k_{s-j}\}$ is the complement of $\{i_1,\dots,i_j\}$ in $\{1,\dots,s\}$ and $p$ is the sign of the permutation $dg_{i_1}\dots dg_{i_j}dg_{k_1}\dots dg_{k_{s-j}}$; see for example \cite[Proposition 17.15]{eisenbudbook}.  It is easy to see that $p=j(s-j)+k_1+\dots+k_{s-j}$. 
\end{proof}

\section{The Self-Dual Complete Resolution}

In this section we construct a self-dual complete resolution of $R/J$ over $R$ under the assumptions introduced at the beginning of the previous section.  
\newline
\indent  We begin our construction by considering the Tate resolution of $R/J$ over $R$, which is known to be a DG-algebra resolution.  By  \cite[Theorem 4]{tate}, since $\underline{f}$ and $\underline{g}$ are both regular sequences on $Q$, this resolution is precisely 
\begin{displaymath}
\mathbb{F}=R\langle dg_1,\dots,dg_s,T_1,\dots,T_r | \thinspace\thinspace\partial (dg_i)=g_i, \thinspace\thinspace\partial (T_j)=\sum_{i=1}^s a_{ji}dg_i\rangle.
\end{displaymath}
As in \cite[Section 3]{HM93}, we view this resolution as the total complex of the double complex 
\begin{align*}
\mathbb{F}_{i,j}=\bigwedge\nolimits^{i} G\otimes D_j F
\end{align*}
where $G$ is the free $R$-module with basis $dg_1,\dots,dg_s$, $\bigwedge G$ denotes the exterior algebra on that basis, $F$ is the free $R$-module with basis $T_1,\dots,T_r$, and $D$ denotes the divided power algebra on that basis.  We observe that the vertical maps in the double complex are given by the differentials of the Koszul complex $K(\underline{g};R)$, and the horizontal maps send $T_j$ to $\sum_{i=1}^s a_{ji}dg_i$.  For ease of notation we set $\bigwedge^i \otimes D_j :=\bigwedge^i G\otimes D_j F$.  
\newline
\indent  Dualizing the double complex $\mathbb{F}$, and we note that  
\begin{align*}
\left(\mathbb{F}_{i,j}\right)^*=\left(\bigwedge\nolimits^{i} G\otimes D_j F\right)^*=\left(\bigwedge\nolimits^{i}G\right)^*\otimes \left(D_jF\right)^*=\bigwedge\nolimits^i G^*\otimes S_jF^*
\end{align*}
where $S$ denotes the symmetric algebra on the dual basis of $F$.  We denote by $\bigwedge^i \otimes S_j :=\bigwedge^i G^*\otimes S_j F^*$, and we note that in $\mathbb{F}^*$ the vertical maps $\mu_x$ are given by right multiplication by $x=\sum_{j=i}^s g_idg^*_i$ and the horizontal maps $d$ send $dg^*_i$ to $\sum_{j=1}^r a_{ji}T_j^*$. 
\newline 
\indent  By \cite[Theorem 2.5]{quasici}, since $J$ is a quasi-complete intersection ideal, it is a \textit{quasi-Gorenstein ideal}; that is, there are isomorphisms
\begin{align}\label{quasigor}
\Ext_R^i(R/J,R)\cong\begin{cases}
R/J & i=g \\
0 & i\neq g
\end{cases}.
\end{align} 
\newline
\indent Following work of Herzog and Martsinkovsky in \cite[Section 3]{HM93} --- see also \cite{quasici} for another gluing construction --- we seek to glue the double complex $\mathbb{F}$ to its dual.  That is, we aim to define a map $v:\mathbb{F}\rightarrow\mathbb{F}^*[-g]$ as shown in diagram (\ref{diagram}) below that identifies the homologies $H_0(\mathbb{F})=R/J$ and 
\begin{align}\label{ext}
H^g(\mathbb{F}^*)=H^g(\mathrm{Hom}_R(\mathbb{F},R))=\mathrm{Ext}_R^g(R/J,R)\cong R/J,
\end{align}
where the isomorphism follows from (\ref{quasigor}).  

{\small{\begin{align}\label{diagram}
\begin{tikzcd}[ampersand replacement=\&, column sep=small, row sep=small]
\, \& \vdots \arrow{d} \& \vdots\arrow{d} \& \vdots\arrow{d} \& \vdots\arrow{d} \& \vdots\arrow{d} \& \vdots\arrow{d} \& \, \\
\dots \arrow{r} \& \bigwedge^1\otimes D_2 \arrow{r}{d^T}\arrow{d}{\mu_x^T} \& \bigwedge^2\otimes D_1 \arrow{r}{d^T}\arrow{d}{\mu_x^T} \& \bigwedge^3\otimes D_0 \arrow{r}{v_3}\arrow{d}{\mu_x^T} \& \bigwedge^{g-3}\otimes S_0 \arrow{r}{d}\arrow{d}{\mu_x} \& \bigwedge^{g-4}\otimes S_1 \arrow{r}{d}\arrow{d}{\mu_x} \& \bigwedge^{g-5}\otimes S_2 \arrow{r}\arrow{d}{\mu_x} \& \dots \\
0 \arrow{r} \& \bigwedge^0\otimes D_2 \arrow{r}{d^T}\arrow{d} \& \bigwedge^1\otimes D_1 \arrow{r}{d^T}\arrow{d}{\mu_x^T} \& \bigwedge^2\otimes D_0 \arrow{r}{v_2}\arrow{d}{\mu_x^T} \& \bigwedge^{g-2}\otimes S_0 \arrow{r}{d}\arrow{d}{\mu_x} \& \bigwedge^{g-3}\otimes S_1 \arrow{r}{d}\arrow{d}{\mu_x} \& \bigwedge^{g-4}\otimes S_2 \arrow{r}\arrow{d}{\mu_x} \& \dots \\
\, \& 0 \arrow{r} \& \bigwedge^0\otimes D_1 \arrow{r}{d^T}\arrow{d} \& \bigwedge^1\otimes D_0 \arrow{r}{v_1}\arrow{d}{\mu_x^T} \& \bigwedge^{g-1}\otimes S_0 \arrow{r}{d}\arrow{d}{\mu_x} \& \bigwedge^{g-2}\otimes S_1 \arrow{r}{d}\arrow{d}{\mu_x} \& \bigwedge^{g-3}\otimes S_2 \arrow{r}\arrow{d}{\mu_x} \& \dots \\
\, \& \, \& 0 \arrow{r} \& \bigwedge^0\otimes D_0 \arrow{r}{v_0}\arrow{d} \& \bigwedge^{g}\otimes S_0 \arrow{r}{d}\arrow{d} \& \bigwedge^{g-1}\otimes S_1 \arrow{r}{d}\arrow{d} \& \bigwedge^{g-2}\otimes S_2 \arrow{r}\arrow{d} \& \dots \\
\, \& \, \& \, \& 0 \& \vdots \& \vdots \& \vdots \& \, 
\end{tikzcd}
\end{align}}}

The diagram shows the resolution $\mathbb{F}$ on the left, its dual on the right, and a gluing map $v$ between them.  Apriori, the gluing map $v$ need not land in the first column as the diagram shows; however the gluing map we define in this section does indeed satisfy this property.

We note that since $\mathbb{F}^*$ is a module over $\mathbb{F}$, to define the gluing map $v$, it suffices to define $v_0$ and take $v$ to be multiplication by the element $v_0(1)$.  We will define $v_0$ carefully in order to construct a complete resolution of $R/J$ over $R$ which is self-dual.  
\newline
\indent
To this end, we denote by $\beta$ the image of $1+J$ under the isomorphism 
\begin{align}\label{eq:26}
R/J\rightarrow\Ext_R^g(R/J,R)
\end{align}
from (\ref{quasigor}).  Given the equality $\mathrm{Ext}_R^g(R/J,R)=H^g(\mathbb{F}^*)$ from (\ref{ext})
we have that 
\begin{align*}
\beta=\alpha+\mathrm{Im}\,\partial
\end{align*}
where $\alpha=(\alpha_0,\dots,\alpha_{\lceil\frac{g+1}{2}\rceil-1})\in\mathrm{Ker}\,\partial\setminus\mathrm{Im}\,\partial$ with each $\alpha_i\in\bigwedge^{g-2i}\otimes S_i$ and where $\partial$ denotes the differential on $\mathbb{F}^*$.  With this notation and the diagram (\ref{diagram}) in mind, we get the following technical lemma, which says that the element $\alpha_0$ is not a boundary in the dual Koszul complex $K(\underline{g};R)$.   

\begin{lemma}\label{notboundary}
The element $\alpha_0$ satisfies the property that $\alpha_0\neq\mu_x(\delta_0)$ for any $\delta_0\in\bigwedge^{g-1}\otimes S_0$.
\end{lemma}

\begin{proof}
Suppose that $\alpha_0=\mu_x(\delta_0)$ for some $\delta_0\in\wedge^{g-1}\otimes S_0$.  Since each side of diagram (\ref{diagram}) commutes, we have that $d(\alpha_0)=\mu_x(d(\delta_0))$. On the other hand, since $\alpha\in\text{Ker}\,\partial$, we have that $\mu_x(\alpha_1)=d(\alpha_0)$.  Thus $d(\delta_0)-\alpha_1\in\text{Ker}\,\mu_x$. 
\newline
\indent  We note that each column on the right half of the diagram (\ref{diagram}) is the dual Koszul complex on $g_1,\dots,g_s$. By self-duality of the Koszul complex, since the last nonvanishing Koszul homology module is $H_{s-g}(\underline{g};R)$, the first nonvanishing Koszul cohomology module is $H^g(\underline{g};R)$.  Thus the columns on the right are exact in degrees lower than $g$, and we have that $d(\delta_0)-\alpha_1=\mu_x(\delta_1)$, for some $\delta_1\in\bigwedge^{g-3}\otimes S_1$. 
\newline
\indent  Inductively, we have elements $\delta_i$, for $i=0,\dots,\lceil\frac{g+1}{2}\rceil-1$, such that each $d(\delta_i)-\alpha_{i+1}=\mu_x(\delta_{i+1})$.  Thus we have 
\begin{align*}
\alpha=\partial(\delta_0,\dots,\delta_{\lceil\frac{g+1}{2}\rceil-1}),
\end{align*}
which is a contradiction.
\end{proof}

Continuing our construction, we note that since $\alpha$ is a cycle in $\mathbb{F}^*$, we must have that $\alpha_0$ is a cycle in the dual Koszul complex; however, it is not a boundary by Lemma \ref{notboundary}.  Thus by Corollary \ref{cohomgens}, we see that 
\begin{align}\label{eq:alpha0}
\alpha_0=t\Bigg(\sum_{1\leq i_1<\dots<i_r\leq s}(-1)^{k_1+\dots+k_g}A_{1,..,r}^{i_1,\dots,i_r}dg^*_{k_1}\dots dg^*_{k_g}\Bigg)+\mu_x(\gamma),
\end{align} 
for some nonzero $t\in R$ and some $\gamma\in\wedge^{g-1}\otimes S_0$.  
\newline
\indent  Now we define the gluing map $v$ as follows.  Define the map $v_0:\mathbb{F}_{0}\rightarrow\mathbb{F}_{g}^*$ such that $v_0(1)=\epsilon$, where
\begin{align}\label{eq:epsilon}
\epsilon=\Bigg(t\Bigg(\sum_{1\leq i_1<\dots<i_r\leq s}(-1)^{k_1+\dots+k_g}A_{1,\dots,r}^{i_1,\dots,i_r}dg^*_{k_1}\dots dg^*_{k_g}\Bigg),0,\dots,0\Bigg),
\end{align}
and extend this to the chain map 
\begin{align}\label{v}
v:\mathbb{F}&\longrightarrow\mathbb{F}^*[-g] \\
\eta&\longmapsto\epsilon\cdot\eta.\nonumber
\end{align}
Thus the gluing map does indeed land in the first column on the right side of the diagram (\ref{diagram}).
\newline
\indent We aim to show that the map $v$ is a quasi-isomorphism.  To accomplish this, we need some technical lemmas.  

\begin{lemma}\label{horizzero}
The element 
\begin{align*}
\sum_{1\leq i_1<\dots<i_r\leq s}(-1)^{k_1+\dots+k_g}A_{1,\dots,r}^{i_1,\dots,i_r}dg^*_{k_1}\dots dg^*_{k_g} 
\end{align*}
is in the kernel of the map $d$.
\end{lemma}

\begin{proof}
For ease of notation, we set $K=k_1+\dots+k_g$, and we denote by $\sum_I$ the sum over $1\leq i_1<\dots<i_r\leq s$.  Applying $d$ to the given element, we see that it suffices to show that the coefficient of $T_m$ in the sum 
\begin{align*}
\sum_I(-1)^{K}A_{1,..,r}^{i_1,\dots,i_r}d(dg^*_{k_1}\dots dg^*_{k_g})=\sum_{I}(-1)^{K}A_{1,..,r}^{i_1,\dots,i_r}\sum_{n=1}^g dg^*_{k_1}\dots\widehat{dg^*_{k_n}}\dots dg^*_{k_g}\sum_{\ell=1}^r a_{\ell,k_n}T_{\ell}
\end{align*}
is zero for all $m=1,\dots,r$.  We observe that its coefficient is precisely
\begin{align*}
\sum_{I}(-1)^{K}A_{1,..,r}^{i_1,\dots,i_r}\sum_{n=1}^g (-1)^{n+1} dg^*_{k_1}\dots\widehat{dg^*_{k_n}}\dots dg^*_{k_g}a_{m,k_n},
\end{align*}
so it suffices to show that the coefficient of any $dg^*_{p_1}\dots dg^*_{p_{g-1}}$ in this sum is zero.  
\newline
\indent  Thus the terms of the sum we are interested in are the ones where we sum over all sets $\{i_1,\dots,i_r\}$ whose complement in $\{1,..,s\}$ contains the set $\{p_1,\dots,p_{g-1}\}$.  There are $s-(g-1)=r+1$ such sets, namely,
\begin{align*}
\{q_2,\dots,q_{r+1}\},\dots,\{q_1,\dots,\widehat{q_i},\dots,q_{r+1}\},\dots,\{q_1,\dots,q_r\},
\end{align*}
and their complements are 
\begin{align*}
\{p_1,\dots,p_{g-1},q_1\},\dots,\{p_1,\dots,p_{g-1},q_{r+1}\}
\end{align*}
respectively, where $\{q_1,\dots,q_{r+1}\}$ is the complement of $\{p_1,\dots,p_{g-1}\}$ in $\{1,\dots,s\}$.  
\newline
\indent  Let $s_i$ denote the spot in which $q_i$ sits when $\{p_1,\dots,p_{g-1},q_i\}$ is ordered.  Thus the coefficient of interest is 
\begin{align*}
\sum_{i=1}^{r+1} (-1)^{p_1+\dots+p_{g-1}+q_i+s_i+1}A_{1,\dots,r}^{q_1,\dots,\widehat{q_i},\dots,q_{r+1}}a_{m,q_i}.
\end{align*}
We note that if the signs in the sum above are alternating, then this coefficient is zero.  Indeed, it would equal the determinant $A_{m,1,\dots,m,\dots,r}^{q_1,\dots,q_{r+1}}$, which is zero as a row is repeated.  
\newline 
\indent  Thus it suffices to show that $(-1)^{q_i+s_i}$ is the opposite of $(-1)^{q_{i+1}+s_{i+1}}$ for any $1\leq i\leq r+1$.  We assume that $q_i$ and $q_{i+1}$ have the same parity, as the argument for the case in which they have different parity is essentially the same.  Note that in this case, there is an odd number of indices $p_j$ between $q_i$ and $q_{i+1}$ in the set $\{1,\dots,s\}$.  Thus $s_{i+1}$ is the sum of $s_i$ and an odd number.  Hence $s_i$ and $s_{i+1}$ have different parity, and $(-1)^{q_i+s_i}$ and $(-1)^{q_{i+1}+s_{i+1}}$ are opposites. Therefore the signs in the sum are alternating as desired. 
\end{proof}

We also need the following lemma of Miller from \cite{miller} in the proof of our next proposition. 

\begin{lemma}\label{millerlemma}
If $\gamma=(\gamma_0,\dots,\gamma_{\lceil\frac{g+1}{2}\rceil-1})\in\mathrm{Ker}\,\partial$ and $d(\gamma_0)=0$, then there is an element $\gamma'=(\gamma_0,0,\dots,0)\in\mathrm{Ker}\,\partial$ such that $[\gamma]=[\gamma']$. 
\end{lemma} 

Given our construction above, we obtain the following result.

\begin{prop}\label{completeresn}
The map $v:\mathbb{F}\rightarrow\mathbb{F}^*[-g]$ defined in (\ref{v}) is a quasi-isomorphism.
\end{prop} 

\begin{proof}
Since $\mathrm{Ext}_R^i(R/J,R)=0$ for all $i\neq g$, it suffices to show that $v_0\colon\mathbb{F}_0\rightarrow\mathbb{F}^*_g$ induces an isomorphism on homology.  However, we observe that the map induced by $v_0$ on homology is precisely the isomorphism $R/J\rightarrow\Ext_R^g(R/J,R)$ in (\ref{eq:26}).  Indeed, since $\beta=\alpha+\text{Im}\,\partial$, we have the following equalities 
\begin{align*}
\beta&=\left(t\left(\sum_{I}(-1)^{K}A_{1,..,r}^{i_1,\dots,i_r}dg^*_{k_1}\dots e^*_{k_g}\right)+\mu_x(\gamma),\alpha_1,\dots,\alpha_{\lceil\frac{g+1}{2}\rceil-1}\right)+\text{Im}\,\partial \\
&=\left(t\left(\sum_{I}(-1)^{K}A_{1,\dots,r}^{i_1,\dots,i_r}dg^*_{k_1}\dots dg^*_{k_g}\right),\alpha_1-d(\gamma),\dots,\alpha_{\lceil\frac{g+1}{2}\rceil-1}\right)+\partial(\gamma,0,\dots,0)+\text{Im}\,\partial \\
&=\left(t\left(\sum_{I}(-1)^{K}A_{1,\dots,r}^{i_1,\dots,i_r}dg^*_{k_1}\dots dg^*_{k_g}\right),\alpha_1-d(\gamma),\dots,\alpha_{\lceil\frac{g+1}{2}\rceil-1}\right)+\text{Im}\,\partial \\
&=\left(t\left(\sum_{I}(-1)^{K}A_{1,\dots,r}^{i_1,\dots,i_r}dg^*_{k_1}\dots dg^*_{k_g}\right),0,\dots,0\right)+\text{Im}\,\partial \\
&=\epsilon+\mathrm{Im}\,\partial
\end{align*}
where $K=k_1+\dots+k_g$ and $\sum_I$ denotes the sum over $1\leq i_1<\dots<i_r\leq s$, and where the first equality follows from (\ref{eq:alpha0}), the second equality follows from the fact that 
\begin{align*}
\partial(\gamma,0,\dots,0)=(\mu_x(\gamma),d(\gamma),0,\dots,0),
\end{align*}
the fourth equality follows from Lemma \ref{horizzero} and Lemma \ref{millerlemma}, and the last follows from (\ref{eq:epsilon}).  Therefore $\beta=v_0(1)+\mathrm{Im}\,\partial$, and $v$ is a quasi-isomorphism as desired. 
\end{proof}

Now we aim to show that the cone of the map constructed in (\ref{v}), is self-dual.  The next proposition gives conditions on any gluing map $\omega:\mathbb{F}\rightarrow\mathbb{F}^*[-g]$, which force its cone to be self-dual. 

\begin{prop}\label{selfdual}
Let $\omega:\mathbb{F}\rightarrow\mathbb{F}^*[-g]$ be the map such that
\begin{align*}
\omega_0(1)=\sum_{\substack{1\leq k_1<\dots<k_n\leq s \\ 1\leq\ell_1<\dots<\ell_{m}\leq r \\ n+2m=g}} c_{k_1,\dots,k_n}^{\ell_1,\dots,\ell_m}dg^*_{k_1}\dots dg^*_{k_n}T^*_{\ell_1}\dots T^*_{\ell_m} 
\end{align*}
and $\omega_i$ is defined by multiplication by $\omega_0(1)$.  If either $c_{k_1,\dots,k_n}^{\ell_1,\dots,\ell_m}=0$ for all $n\equiv 0,1\thinspace\mathrm{mod}\thinspace 4$ or for all $n\equiv 2,3\thinspace\mathrm{mod}\thinspace 4$ , then cone($\omega$) is self-dual.
\end{prop}

\begin{proof}
We first show that $\omega_j^T=\pm\omega_{g-j}$ for all $j$, where we view each $\omega_j$ as a matrix and $(\,\underline{\hspace{11pt}}\,)^T$ denotes the transpose.  Note that for $t+2u=j$, we have
\begin{align*}
\omega_j&\left(dg_{p_1}\dots dg_{p_t}T_{q_1}\dots T_{q_u}\right)= \\
&=\left(dg_{p_1}\dots dg_{p_t}T_{q_1}\dots T_{q_u}\right)\cdot\left(\sum_{\substack{1\leq k_1<\dots<k_n\leq s \\ 1\leq\ell_1<\dots<\ell_{m}\leq r \\ n+2m=g}} c_{k_1,\dots,k_n}^{\ell_1,\dots,\ell_m}dg^*_{k_1}\dots e^*_{k_n}T^*_{\ell_1}\dots T^*_{\ell_m}\right) \\
&=\sum_{\substack{1\leq k_1<\dots<k_n\leq s \\ 1\leq\ell_1<\dots<\ell_{m}\leq r \\ n+2m=g}} c_{k_1,\dots,k_n}^{\ell_1,\dots,\ell_m}dg_{p_1}\dots dg_{p_t}T_{q_1}\dots T_{q_u}dg^*_{k_1}\dots dg^*_{k_n}T^*_{\ell_1}\dots T^*_{\ell_m},
\end{align*}
and we observe that $dg_{p_1}\dots dg_{p_t}T_{q_1}\dots T_{q_u}dg^*_{k_1}\dots dg^*_{k_n}T^*_{\ell_1}\dots T^*_{\ell_m}$ is given by
\begin{align}\label{eq:piecewise}
\begin{cases} 
      (-1)^zdg^*_{v_1}\dots dg^*_{v_{n-t}}T^*_{x_1}\dots T^*_{x_{m-u}} & dg_{p_1}\dots dg_{p_t}T_{q_1}\dots T_{q_u}|dg_{k_1}\dots dg_{k_n}T_{\ell_1}\dots T_{\ell_m}  \\
      0 & \text{otherwise} \\
   \end{cases}
\end{align}
for some $z\in\mathbb{Z}$, where $\{v_1,\dots,v_{n-t}\}$ is the complement of $\{p_1,\dots,p_{t}\}$ in $\{k_1,\dots,k_{n}\}$ and $\{x_1,\dots,x_{m-u}\}$ is the complement of $\{q_1,\dots,q_{u}\}$ in $\{\ell_1,\dots,\ell_{m}\}$.  
\newline 
\indent Similarly we have that $\omega_{g-j}\big(dg_{v_1}\dots dg_{v_{n-t}}T_{x_1}\dots T_{x_{m-u}}\big)$ is given by
\begin{align}\label{omegagminusj}
\sum_{\substack{1\leq k_1<\dots<k_n\leq s \\ 1\leq\ell_1<\dots<\ell_{m}\leq r \\ n+2m=g}} c_{k_1,\dots,k_n}^{\ell_1,\dots,\ell_m}dg_{v_1}\dots dg_{v_{n-t}}T_{x_1}\dots T_{x_{m-u}}dg^*_{k_1}\dots dg^*_{k_n}T^*_{\ell_1}\dots T^*_{\ell_m}.
\end{align}
From (\ref{eq:piecewise}), in the case that $dg_{p_1}\dots dg_{p_t}T_{q_1}\dots T_{q_u}|dg_{k_1}\dots dg_{k_n}T_{\ell_1}\dots T_{\ell_m}$, we see that
\begin{align}\label{messyequality}
&dg^*_{k_1}\dots dg^*_{k_n}T^*_{\ell_1}\dots T^*_{\ell_m}=(-1)^zT^*_{q_u}\dots T^*_{q_1}dg^*_{p_t}\dots dg^*_{p_1}dg^*_{v_1}\dots dg^*_{v_{n-t}}T^*_{x_1}\dots T^*_{x_{m-u}}.
\end{align}
To show that $\omega_j^T=\pm\omega_{g-j}$ we apply (\ref{messyequality}) to simplify (\ref{omegagminusj}) as follows
\begin{align*}
dg_{v_1}&\dots dg_{v_{n-t}}T_{x_1}\dots T_{x_{m-u}}dg^*_{k_1}\dots dg^*_{k_n}T^*_{\ell_1}\dots T^*_{\ell_m}=  \\
&=(-1)^zdg_{v_1}\dots dg_{v_{n-t}}T_{x_1}\dots T_{x_{m-u}}T^*_{q_u}\dots T^*_{q_1}dg^*_{p_t}\dots dg^*_{p_1}dg^*_{v_1}\dots dg^*_{v_{n-t}}T^*_{x_1}\dots T^*_{x_{m-u}} \\
&=(-1)^zdg_{v_1}\dots dg_{v_{n-t}}dg^*_{p_t}\dots dg^*_{p_1}T^*_{q_1}\dots T^*_{q_u}dg^*_{v_1}\dots dg^*_{v_{n-t}} \\
&=(-1)^z(-1)^{(n-t)(\frac{t+n-1}{2})}dg^*_{p_t}\dots dg^*_{p_1}T^*_{q_1}\dots T^*_{q_u}
\end{align*} 
where the last two equalities follow from reordering the terms and simplifying each $T_{x_i} T^*_{x_i}$ and $dg_{v_i} dg^*_{v_i}$.  Reordering the $dg^*_{k_i}$ terms and keeping track of the resulting sign changes, we have that 
\begin{align*}
dg_{v_1}&\dots dg_{v_{n-t}}T_{x_1}\dots T_{x_{m-u}}dg^*_{k_1}\dots dg^*_{k_n}T^*_{\ell_1}\dots T^*_{\ell_m}=  \\
&=(-1)^z(-1)^{t(n-t)+\frac{(n-t)(n-t-1)}{2}+\frac{t(t-1)}{2}}dg^*_{p_1}\dots dg^*_{p_t}T^*_{q_1}\dots T^*_{q_u} \\
&=(-1)^z(-1)^{\frac{n(n-1)}{2}}dg^*_{p_1}\dots dg^*_{p_t}T^*_{q_1}\dots T^*_{q_u} \\
&=\begin{cases} 
      (-1)^zdg^*_{p_1}\dots dg^*_{p_t}T^*_{q_1}\dots T^*_{q_u} & n\equiv 0,1\thinspace\text{mod}\thinspace 4   \\
      (-1)^{z+1}dg^*_{p_1}\dots dg^*_{p_t}T^*_{q_1}\dots T^*_{q_u} & n\equiv 2,3\thinspace\text{mod}\thinspace 4   \\
         \end{cases},
\end{align*}
so that $\omega_{g-j}\big(dg_{v_1}\dots dg_{v_{n-t}}T_{x_1}\dots T_{x_{m-u}}\big)$ is given by
\begin{align}\label{omegagminusj2}
&=\begin{cases} 
      \sum c_{k_1,\dots,k_n}^{\ell_1,\dots,\ell_m}(-1)^zdg^*_{p_1}\dots dg^*_{p_t}T^*_{q_1}\dots T^*_{q_u} & n\equiv 0,1\thinspace\text{mod}\thinspace 4   \\ \\
      \sum c_{k_1,\dots,k_n}^{\ell_1,\dots,\ell_m}(-1)^{z+1}dg^*_{p_1}\dots dg^*_{p_t}T^*_{q_1}\dots T^*_{q_u} & n\equiv 2,3\thinspace\text{mod}\thinspace 4   \\
         \end{cases},
\end{align}
where we have dropped the indices on the sum from (\ref{omegagminusj}) for ease of notation.  Comparing our computation for $\omega_{g-j}$ in (\ref{omegagminusj2}) to our computation for $\omega_j$ in (\ref{eq:piecewise}) and applying our assumption on $c_{k_1,\dots,k_n}^{\ell_1,\dots,\ell_m}$, we see that either $\omega_{g-j}=\omega_j^T$ for all $j$ or $\omega_{g-j}=-\omega_j^T$ for all $j$.
\newline
\indent Now $T:=\,\mathrm{cone}(\omega)$ is the complex 
\begin{align}\label{eq:cone}
\dots\longrightarrow \mathbb{F}_{g+1}\overset{D_{g+1}}\longrightarrow \mathbb{F}_g\overset{D_g}\longrightarrow \mathbb{F}_0^*\oplus \mathbb{F}_{g-1}\longrightarrow\dots\longrightarrow \mathbb{F}_{g-2}^*\oplus \mathbb{F}_1\overset{D_1}\longrightarrow \mathbb{F}_{g-1}^*\oplus \mathbb{F}_0\overset{D_0}\longrightarrow \mathbb{F}_g^*\longrightarrow\dots
\end{align}
with differentials given by the block matrices
\begin{align*}
D_j=
\left[
\begin{array}{c|c}
\partial_{g-j}^T & \omega_j \\
\hline
0 & -\partial_j
\end{array}
\right].
\end{align*}
We define $\phi\colon T\longrightarrow T^*[g-1]$, shown in the diagram below, as follows.
\begin{center}
\begin{tikzcd}[column sep=small]
\cdots\arrow{r} & F_{g+1} \arrow{d}{\phi_{g+1}} \arrow{r}{D_{g+1}} & F_g \arrow{r}{D_g} \arrow{d}{\phi_g} & F_0^*\oplus F_{g-1} \arrow{r}{D_{g-1}} \arrow{d}{\phi_{g-1}} &  \cdots \arrow{r}{D_2} & F_{g-2}^*\oplus F_1 \arrow{d}{\phi_1} \arrow{r}{D_1} & F_{g-1}^*\oplus F_0 \arrow{d}{\phi_0} \arrow{r}{D_0} & F_g^* \arrow{r}\arrow{d}{\phi_{-1}} &\cdots\\
 \cdots \arrow{r} & F_{g+1} \arrow{r}{D_{-1}^T} & F_g \arrow{r}{D_0^T} & F_{g-1}\oplus F_0^*  \arrow{r}{D_{1}^T}  &  \cdots \arrow{r}{D_{g-2}^T} & F_1\oplus F_{g-2}^*  \arrow{r}{D_{g-1}^T} &  F_0\oplus F_{g-1}^*  \arrow{r}{D_g^T} & F_g^* \arrow{r} &\cdots 
\end{tikzcd}
\end{center}
In the case that $\omega_{g-j}=\omega_j^T$ for all $j$, we define $\phi$ by
\begin{align*}
\phi_j=\begin{cases} 
      {\left[
\begin{array}{c|c}
0 & -\text{Id} \\
\hline
\text{Id} & 0
\end{array}
\right]} & \text{for} \thinspace\thinspace j \thinspace\thinspace\text{even}   \vspace{0.3cm}\\
      {\left[
\begin{array}{c|c}
0 & \text{Id} \\
\hline
-\text{Id} & 0
\end{array}
\right]} & \text{for} \thinspace\thinspace j \thinspace\thinspace\text{odd}   \\
   \end{cases},
\end{align*}
and in the case that $\omega_{g-j}=-\omega_j^T$ for all $j$, we define $\phi$ by 
\begin{align*}
\phi_j=\begin{cases} 
      {\left[
\begin{array}{c|c}
0 & \text{Id} \\
\hline
\text{Id} & 0
\end{array}
\right]} & \text{for} \thinspace\thinspace j \thinspace\thinspace\text{even}   \vspace{0.3cm}\\
      {\left[
\begin{array}{c|c}
0 & -\text{Id} \\
\hline
-\text{Id} & 0
\end{array}
\right]} & \text{for} \thinspace\thinspace j \thinspace\thinspace\text{odd}   \\
   \end{cases}.
\end{align*}
In either case, it is easy to see that the diagram commutes, and that $\phi$ is an isomorphism between cone($\omega$) and its shifted dual. 
\end{proof}

Recall that our gluing map $v$ defined in (\ref{v}) lands in the first column of the right side of diagram (\ref{diagram}) and thus automatically satisfies the condition on $c_{k_1,\dots,k_n}^{\ell_1,\dots,\ell_m}$ in Proposition \ref{selfdual}; although, this need not be the case for any gluing map $\omega_0$ as in the proposition since it may land in multiple columns.  

To complete our construction, we let 
\begin{align}\label{T}
\mathbb{T}:=\,\mathrm{cone}(v), 
\end{align}
where $v:\mathbb{F}\longrightarrow\mathbb{F}^*[-g]$ is the map defined in (\ref{v}).  We note that by Proposition \ref{completeresn}, $\mathbb{T}$ is exact, and its dual is also exact.  Indeed, by Proposition \ref{selfdual}, it is self-dual.  The fact that its truncation $\mathbb{T}_{\geq g}$ is the tail of a free resolution of $R/J$ is clear from (\ref{eq:cone}).  Thus we have proved the following result.

\begin{thm}\label{completeresnagain}
$\mathbb{T}$ is a self-dual complete resolution of $R/J$ over $R$. \hfill\qedsymbol
\end{thm}

In fact, the resolution we have constructed is the minimal complete resolution of $R/J$, as we see in the following remark.

\begin{remk}
We note that the generators of $I$ and $J$ can be chosen so that the $a_{ji}$ as defined in (\ref{aji}) are in the maximal ideal.  Indeed, if 
\begin{align*}
f_j=\sum_{i\in \Gamma^c}u_ig_i+\sum_{i\in \Gamma} a_{ji}g_i
\end{align*}
where $\Gamma\subseteq \{1,\dots,s\}$, each $u_i$ is a unit, and each $a_{ji}\in\mathfrak{m}$, then
\begin{align*}
f_j+\mathfrak{m}(\underline{g})=\sum_{i\in \Gamma^c}u_ig_i+\mathfrak{m}(\underline{g}); 
\end{align*}
thus by Nakayama's Lemma, we can replace a generator $g_k$, for some $k\in \Gamma^c$, by $f_j$ in our minimal generating set for $(\underline{g})$.  However, 
\begin{align*}
R/J=R/(\bar{g_1},\dots,\bar{f_j},\dots,\bar{g_s})=Q/(f_1,\dots,f_r,g_1,\dots,\widehat{g_k},\dots g_s)=R/(\bar{g_1},\dots,\widehat{\bar{g_k}},\dots,\bar{g_s}), 
\end{align*}
so we may remove this generator from the ideal $(\underline{g})$ in our construction of the resolution of $R/J$.  Repeating this argument as necessary, we see that we may remove all of the generators $g_i$ with $i\in \Gamma^c$, and thus we may assume that all of the $a_{ji}$ are in the maximal ideal in our construction.   
\newline
\indent  By the discussion above, we see from (\ref{eq:epsilon}) that the map $v$ defined in (\ref{v}) is minimal, and thus its cone $\mathbb{T}$ is also minimal.  
\end{remk}  

Applying Proposition \ref{torext} and Theorem \ref{completeresnagain}, we arrive at the following Corollary.  

\begin{cor}
Let $Q$ be a Noetherian local ring and let $I=(\underline{f})\subseteq(\underline{g})$ be embedded complete intersection ideals in $Q$.  If $R=Q/I$ and $J=(\underline{g})/I\subseteq R$, then for any $R$-module $N$, there are isomorphisms
\begin{align*}
\widehat{\Tor}_n^R(R/J,N)\cong\widehat{\Ext}_R^{n+g-1}(R/J,N)
\end{align*}
for all $n\in\mathbb{Z}$.  \hfill\qedsymbol
\end{cor}

\vspace{0.3cm}

\section*{\centering\normalsize{\normalfont{\uppercase{Acknowledgments}}}}

The author would like to thank the organizers of the \textit{Stable Cohomology: Foundations and Applications} workshop held in Snowbird, UT in the Spring of 2018 for the opportunity to speak about these results, and the attendees of the workshop for their helpful and insightful comments.  She would also like to thank her advisor, Claudia Miller, for her support and guidance throughout the project.  Finally, she thanks the anonymous referee for the insightful comments that have greatly improved the paper.   

\vspace{0.3cm}

\bibliographystyle{siam}
\bibliography{selfdual}

\vspace{1.25cm}

\noindent\textit{Affiliation:} \textsc{Department of Mathematics, Syracuse University, Syracuse, NY}
\vspace{2pt}
\newline
\noindent\textit{Current Address:} \textsc{Department of Mathematics, Yale University, New Haven, CT}
\vspace{2pt}
\newline
\textit{E-mail Address}: rachel.diethorn@yale.edu

 \end{document}